\documentclass[a4paper,11pt]{article}
\usepackage[]{graphicx}
\usepackage[]{amsmath,amssymb}
\usepackage[]{amsthm,enumerate}
\usepackage{verbatim,bm,comment}
\usepackage{color}
\usepackage{comment}
\usepackage{a4wide}
\usepackage{lscape,float}

\newtheorem{theorem}{Theorem}[section]
\newtheorem{lemma}[theorem]{Lemma}
\newtheorem{proposition}[theorem]{Proposition}
\newtheorem{corollary}[theorem]{Corollary}
\theoremstyle{definition}

\newtheorem{notation}[theorem]{Notation}
\newtheorem{example}[theorem]{Example}
\newcommand{\defn}[1]{{\em #1}}
\theoremstyle{remark}
\newtheorem{remark}[theorem]{Remark}
\newtheorem{problem}[theorem]{Problem}

\title{Linked systems of symmetric group divisible designs of type II} 
\date{
\today
}

\author{
 Hadi Kharaghani\thanks{Department of Mathematics and Computer Science, University of Lethbridge,
Lethbridge, Alberta, T1K 3M4, Canada. \texttt{kharaghani@uleth.ca}} 
\and  
 Sho Suda\thanks{Department of Mathematics Education,  Aichi University of Education, 1 Hirosawa, Igaya-cho,  Kariya, Aichi, 448-8542, Japan. \texttt{suda@auecc.aichi-edu.ac.jp}}
}

\begin{document}
\maketitle

%%%%%%%%%%%%%%%%%%%%%%%%%%%%%%%%%%%%%%%%%%%%%%%%%%%%%%%%%%%%%%%%%%%%%%%%%%%%%%%%%%%%%%%%%%%%%%%%%%%%%%%%%%%%%%%%%%%
\abstract{
%The  %a collection of symmetric group divisible designs which is called 
 Linked systems of symmetric group divisible designs of type II is introduced, and several examples are obtained from affine resolvable designs and a variant of
 Mutually Orthogonal Latin Squares (MOLS). 
Furthermore, an equivalence between such symmetric group divisible designs and some association schemes with $5$-classes is provided. 
}



\section{Introduction}

This paper is a continuation of the paper \cite{KS}.  Therein the authors introduced the concept  of a \defn{linked system of symmetric group divisible designs} to be a collection of symmetric group divisible designs $A_{i,j}$ ($i,j\in\{1,\ldots,f\}$, $f\geq 3$) satisfying the following property; there exist non-negative integers $\sigma,\tau$ such that for any mutually distinct $i,j$, and $l$, $A_{i,j}A_{j,l}=\sigma A_{i,l}+\tau(J_v-A_{i,l})$ holds. 
This concept includes mutually unbiased biangular vectors \cite{HKS}, mutually unbiased real Hadamard matrices \cite{LMO}, as examples. 
Among other results, it was shown that a linked system of symmetric group divisible designs is equivalent to some association schemes with $4$-classes. For the details of some of the concepts not defined in this note, we refer the reader to \cite{KS}. 

Throughout this work, $I_n,J_n,O_n$ denote the $n\times n$ identity matrix, all-ones matrix,  and zero matrix, respectively.  
We omit the subscript when the order is clear from the context. 

Fixing a suitable ordering of points and blocks, (see \cite{KS}), a $v\times v$ $(0,1)$-matrix $A$ is the incidence matrix of a \emph{symmetric group divisible design} if $A A^\top=A^\top A=k I_v+\lambda_1(K_{m,n}-I_v)+\lambda_2(J_v-K_{m,n})$ holds for some non-negative integers $k,\lambda_1,\lambda_2$ where $v=mn$ and $K_{m,n}=I_m\otimes J_n$ and $A^\top$ denotes the transpose of $A$.

The purpose of this article is to introduce a different variant of class of symmetric group divisible designs, which we call  \emph{linked systems of symmetric group divisible designs of type II}. 
These are collection of symmetric group divisible designs $A_{i,j}$ ($i,j\in\{1,\ldots,f\}$, $f\geq 3$) satisfying the following properties; $A_{i,j}+K_{m,n}$ is a $(0,1)$-matrix for any distinct $i,j$, and there exist non-negative integers $\sigma,\tau,\rho$ such that for any mutually distinct $i,j$ and $l$, $A_{i,j}A_{j,l}=\sigma A_{i,l}+\tau (J_v-A_{i,l}-K_{m,n})+\rho K_{m,n}$ holds. %See the precise definition in Section~\ref{sec:lsgdd}. 

The paper \cite{KS} dealt with symmetric group divisible designs $A$ for which $A K_{m,n}=K_{m,n}A$ is a multiple of $J_v$, while the present paper deals with symmetric group divisible designs $A$ satisfying  $A K_{m,n}=K_{m,n}A$ which is a multiple of $J_v-K_{m,n}$.  

%A result in this paper is closely related to a part of the results in \cite{KSS}. 
%In \cite{KSS}, 
Kharaghani, Sasani, and Suda \cite{KSS} introduced the concept of mutually unbiased Bush-type Hadamard matrices, and provided a construction using a Hadamard matrix and a set of MOLS.  
As will be shown in Section~\ref{sec:mubh}, these matrices give rise to Hadamard designs and an example of linked systems of symmetric group divisible designs of type II.   
%Note that mutually unbiased Bush-type Hadamard matrices form a special type of real mutually unbiased bases, see \cite{LMO} for the details and a connection to association schemes. 
Note that Hadamard designs are equivalent to affine resolvable designs, see Section~\ref{sec:ard} and \cite{HBCD}.  
Our results in this paper is applied to any affine resolvable design.
%is an extension the results on Hadamard designs to any affine resolvable designs. 
We also establish an equivalence between a linked system of symmetric group divisible designs of type II and specific association scheme with $5$-classes. 
Note that a suitable use of an affine resolvable design in a Latin square yields a symmetric design \cite{W}, see also \cite[Theorem~5.23]{S} and that linked systems of symmetric group divisible designs obtained from finite fields lead to an association scheme \cite{KScom}.

The organization of this paper is as follows.  
In Section~2, we prepare fundamental results on symmetric group divisible designs,  affine resolvable designs, and Latin squares. 
In Section~3, we study symmetric group divisible designs $A$ satisfying the property that $A+K_{m,n}$ is also a symmetric group divisible design.  
%A characterization of Bush-type Hadamard matrices by symmetric group divisible designs follows as a corollary to the results of this section.
%As a corollary of results presented in this section, Bush-type Hadamard matrices are characterized by symmetric group divisible designs.  
In Section~4, we introduce the concept of linked system of symmetric group divisible designs of type II, and derive some necessary conditions to conclude some restrictions on  the parameters of such objects. 
In Section~5, we show the equivalence of existence of some association schemes with $5$-classes and that of a linked system of symmetric group divisible designs of type II. This enables us to 
find an upper bound on the number of a linked system of symmetric group divisible designs of type II.  
In Section~6 and Section~7, we construct symmetric group divisible designs $A$ with the property that  $A K_{m,n}=K_{m,n}A$ is a multiple of $J_v-K_{m,n}$, and linked systems of symmetric group divisible designs of type II, which are based on affine resolvable designs and some specific Latin squares. Our main reference, unless otherwise specifically mentioned, is Part IV of the Handbook of Combinatorial Designs \cite{HBCD}.

%%%%%%%%%%%%%%%%%%%%%%%%%%%%%%%%%%%%%%%%%%%%%%%%%%%%%%%%%%%%%%%%%%%%%%%%%%%%%%%%%%%%%%%%%%%%%%%%%%%%%%%%%%%%%%%%%%%
\section{Preliminaries}
\subsection{Association schemes}
We follow \cite{BI} as our reference for this section. 
A \emph{symmetric association scheme with $d$-classes} %, %see \cite{BI},
with a finite vertex set $X$,  
is a set of non-zero $(0,1)$-matrices $A_0, A_1,\ldots, A_d$ with
rows and columns indexed by $X$, such that
\begin{enumerate}
\item $A_0=I_{|X|}$,
\item $\sum_{i=0}^d A_i = J_{|X|}$,
\item $A_i^\top=A_i$ for $i\in\{1,\ldots,d\}$ and, 
\item for any $i$, $j$ and $k$, there exist non-negative integers $p_{i,j}^k$ such that $A_iA_j=\sum_{k=0}^d p_{i,j}^k A_k$.
\end{enumerate}
The \defn{intersection matrix} $B_i$ is defined to be $B_i=(p_{i,j}^k)_{j,k=0}^{d}$. 
Since each $A_i$ is symmetric, it follows from the condition (iv) that the $A_i$'s necessarily commute.
The vector space spanned by the $A_i$'s over the real number field forms a commutative algebra, denoted by $\mathcal{A}$ and is called the \emph{Bose-Mesner algebra} or \emph{adjacency algebra}.
There exists also a basis of $\mathcal{A}$ consisting of primitive idempotents, say $E_0=(1/|X|)J_{|X|},E_1,\ldots,E_d$. 
Since  $\{A_0,A_1,\ldots,A_d\}$ and $\{E_0,E_1,\ldots,E_d\}$ are two bases of $\mathcal{A}$, there exist the change-of-basis matrices $P=(P_{i,j})_{i,j=0}^d$, $Q=(Q_{i,j})_{i,j=0}^d$ so that
\begin{align*}
A_i=\sum_{j=0}^d P_{j,i}E_j,\quad E_j =\frac{1}{|X|}\sum_{i=0}^{d} Q_{i,j}A_i.
\end{align*}   
The matrices $P,Q$ are said to be the \defn{first and second eigenmatrices} respectively.  
Since $(0,1)$-matrices $A_i$ ($0\leq i\leq d$) have disjoint support, the algebra $\mathcal{A}$ they form is closed under entrywise multiplication denoted by $\circ$. 
The \emph{Krein parameters} $q_{i,j}^k$ are defined by $E_i\circ E_j=\frac{1}{|X|}\sum_{k=0}^dq_{i,j}^k E_k$. 
The \emph{Krein matrix} $B_i^*$ is defined as $B_i^*=(q_{i,j}^k)_{j,k=0}^d$. 
The rank of $E_i$ is denoted by $m_i$ and is called the \emph{$i^{\rm th}$ multiplicity}.
The multiplicity $m_i$ is equal to $Q_{0,i}$. 

The following property of positivity of Krein parameters is very useful. 
\begin{proposition}\label{prop:krein}\upshape{\cite[Theorem~3.8]{BI}}
For any $i,j,k\in\{0,1,\ldots,d\}$, $q_{i,j}^k\geq0$ holds.
\end{proposition}
%\begin{proof}
%Taking the rank of $E_i\circ E_j=\frac{1}{|X|}\sum_{l=0}^d q_{i,j}^l E_l$ for $i,j\in \{0,1\ldots,d\}$, we have 
%\begin{align*}
%\sum_{l:q_{i,j}^l}m_l=\mathrm{rank}
%\end{align*} 
%\end{proof}

Each $A_i$
can be considered as the adjacency matrix of some undirected simple graph. 
The scheme is \emph{imprimitive} if,
on viewing the $A_i$ as adjacency matrices
of graphs $G_i$ on vertex set $X$, at least one of the $G_i$, $i \ne 0$, is disconnected. 
In this case, there exists a set $\mathcal{I}$ of indices such that $0$ and such $i$ are elements of $\mathcal{I}$ and $\sum_{j\in\mathcal{I}}A_j=I_p\otimes J_q$ for some $p,q$ with $p>1$ after suitably permuting the vertices. 
Thus the set $X$ is partitioned into $p$ subsets called \emph{fibers}, each of which has size $q$. 
The set $\mathcal{I}$ defines an equivalence relation on $\{0,1,\ldots,d\}$ by $j\sim k$ if and only if $p_{i,j}^k\neq 0$ for some $i\in \mathcal{I}$.  
Let $\mathcal{I}_0=\mathcal{I},\mathcal{I}_1,\ldots,\mathcal{I}_t$ be the equivalence classes on $\{0,1,\ldots,d\}$ by $\sim$. 
Then by \cite[Theorem~9.4]{BI} there exist $(0,1)$-matrices $\overline{A}_j$ ($0\leq j\leq t$) such that 
\begin{align*}
\sum_{i\in \mathcal{I}_j}A_i=\overline{A}_j\otimes J_q,
\end{align*}
and the matrices $\overline{A}_j$ ($0\leq j\leq t$) define an association scheme on the set of fibers. 
This is called the \emph{quotient association scheme} with respect to $\mathcal{I}$.

For fibers $U$ and $V$, let $\mathcal{I}(U,V)$ denote the set of indices $i$ of adjacency matrices $A_i$ such that 
an entry of $A_i$ of a row indexed by $U$ and a column indexed by $V$ is one. 
For $i\in \mathcal{I}(U,V)$, we define a $(0,1)$-matrix $A_i^{UV}$ by 
\begin{align*}
(A_i^{UV})_{xy}=\begin{cases}
1 &\text{ if } (A_i)_{xy}=1, x\in U, y\in V,\\
0 &\text{ otherwise}.
\end{cases}
\end{align*}
%\begin{definition}
An imprimitive association scheme is called \defn{uniform} if its quotient association scheme is with $1$-class and there exist non-negative integers $a_{i,j}^k$ such that for  all fibers $U,V,W$ and $i\in\mathcal{I}(U,V),j\in\mathcal{I}(V,W)$, we have 
\begin{align*}
A_i^{UV}A_i^{VW}=\sum_{k}a_{i,j}^k A_k^{UW}.
\end{align*}
%\end{definition}

\subsection{Symmetric group divisible designs}
Let $m$ and $n$ be positive integers both not smaller than two. 
A \emph{{\rm(}square{\rm)} group divisible design with parameters $(v,k,m,n,\lambda_1,\lambda_2)$} is a pair $(V,\mathcal{B})$, where $V$ is a finite set of  $v$ elements called points, and $\mathcal{B}$ a collection of $k$-element subsets of $V$ called blocks with $|\mathcal{B}|=v$, in which the point set $V$ is partitioned into $m$ classes of size $n$, such that two distinct points from one class occur
together in $\lambda_1$ blocks, and two points from different classes occur together in exactly $\lambda_2$ blocks. 
A group divisible design is said to be \emph{symmetric} (or to have the \emph{dual property}) if its dual, that is the incidence structure interchanging the roles of points and blocks, is again a group divisible design with the same parameters. 
Throughout this paper, we always assume that $k>\lambda_1$\footnote{If $k=\lambda_1$, then its incidence matrix $A$ satisfies that $A=\bar{A}\otimes J_n$ for some incidence matrix $\bar{A}$ of a symmetric design.}. 

The \defn{incidence matrix} of a symmetric group divisible design is a $v\times v$ $(0,1)$-matrix $A$ with rows and columns indexed by the elements of $\mathcal{B}$ and $V$ respectively such that for $x\in V,B\in \mathcal{B}$
\begin{align*}
A_{B,x}=\begin{cases}
1 & \text{ if } x\in B,\\
0 & \text{ if } x\not\in B.
\end{cases}
\end{align*}
Then, after reordering the elements of $V$ and $\mathcal{B}$ appropriately,  
\begin{align}\label{eq:gdd}
A A^\top=A^\top A=k I_v+\lambda_1(K_{m,n}-I_v)+\lambda_2(J_v-K_{m,n}),  
\end{align}
where $K_{m,n}=I_m\otimes J_n$. 
Then the $v\times v$ $(0,1)$-matrix $A$ is the incidence matrix of a symmetric group divisible design if and only if  the matrix $A$ satisfies Equation \eqref{eq:gdd}. 
Thus we also refer to the $v\times v$ $(0,1)$-matrix $A$ satisfying  Equation~\eqref{eq:gdd} as a symmetric group divisible design.  
When $\lambda_1=\lambda_2=:\lambda$, a symmetric group divisible design is said to be a \emph{symmetric $2$-$(v,k,\lambda)$ design}.

Let $A$ be a symmetric group divisible design with parameters $(v,k,m,n,\lambda_1,\lambda_2)$. 
Then it follows that $A J_v=J_v A =k J_v$ and from Equation~\eqref{eq:gdd} with the previous equalities that
\begin{align}\label{eq:p1}
k^2=k+\lambda_1(n-1)+\lambda_2(v-n).
\end{align}

The following is due to a result of Bose \cite[Theorem~2.1]{B}. 
\begin{lemma}\label{lem:b}
Let $A$ be a symmetric group divisible design with parameters $(v,k,m,n,\lambda_1,\lambda_2)$.  
If $\lambda_1\neq \lambda_2$, then $AK_{m,n}A^\top=(n(\lambda_1-\lambda_2)+k-\lambda_1)K_{m,n}+n \lambda_2 J_v$. 
\end{lemma}

\subsection{Affine resolvable designs}\label{sec:ard} 
A \defn{$2$-$(v,b,r,k,\lambda)$ design} or simply a \defn{$2$-design} is a pair $(V,\mathcal{B})$ where $V$ is a $v$-set called the vertex set and $\mathcal{B}$ is a collection of $b$ $k$-subsets of $V$ such that each element of $v$ is contained in exactly $r$ blocks and any $2$-subset of $V$ is contained in exactly $\lambda$ blocks. 
A \emph{parallel class} in a $2$-design is a set of blocks that partitions the vertex set $V$. 
A \defn{resolvable design} is a $2$-design whose blocks can be partitioned into parallel classes. 
An \defn{affine resolvable design} is a resolvable $2$-design  with the property that any two blocks in different parallel classes intersect in a constant number of points. 

Let $(V,\mathcal{B})$ be an affine resolvable $2$-$(v,b,r,k,\lambda)$ design with parallel classes $\Pi_1,\ldots,\Pi_r$. 
Set $\mu=k^2/v$. 
For $i\in\{1,\ldots,r\}$, we define a $v\times v$ $(0,1)$-matrix $C_i$ indexed by the elements of $V$ so that $(x,y)$-entry is $1$ if there exists a block $B\in \Pi_i$ such that $x,y\in B$, $0$ otherwise. 
Then we readily obtain the following.   
\begin{lemma}\label{lem:ard1}
\begin{enumerate}
\item $\sum_{i=1}^{r}C_i=(r-\lambda) I_{v}+\lambda J_{v}$ holds. 
\item For any $i\in\{1,\ldots,r\}$, $C_i C_i^\top=k C_i$ holds. 
\item For any distinct $i,j\in\{1,\ldots,r\}$, $C_iC_j^\top=\mu J_{v}$ holds.
\end{enumerate}
\end{lemma}

Conversely,  we will characterize affine resolvable designs in terms with $(0,1)$-matrices $C_1,\ldots,C_r$ satisfying Lemma~\ref{lem:ard1}. 
Let $v,k,r,\lambda,\mu$ be positive integers,
and assume that $C_1,\ldots,C_r$ are $v\times v$ $(0,1)$-matrices such that 
\begin{enumerate}
\item $\sum_{i=1}^{r}C_i=(r-\lambda) I_{v}+\lambda J_{v}$ holds,  
\item for any $i\in\{1,\ldots,r\}$, $C_iC_i^\top=k C_i$ holds, and 
\item for any distinct $i,j\in\{1,\ldots,r\}$, $C_i C_{j}^\top=\mu J_{v}$ holds.
\end{enumerate}
We call $C_1,\ldots,C_r$ the \emph{auxiliary matrices}.
Furthermore, we set $C_0=O_v$. 
We will now show that these matrices imply the existence of an affine resolvable design. 
First we derive the formula for the parameters. 
The property (i) implies that each $C_i$ has entries $1$ on diagonal.  
Thus this with the property (ii) shows that each row of $C_i$ has constant sums, namely $C_i J_v=J_v C_i=k J_v$ for any $i$. 
Multiplying the equation in (i) by $J_v$ yields $rk J_v=(r-\lambda+\lambda v)J_v$, namely 
\begin{align}\label{eq:p1}
r k=r-\lambda+\lambda v. 
\end{align}
Similarly, multiplying the equation (iii) by $J_v$ yields 
\begin{align}\label{eq:p2}
k^2=\mu v.
\end{align}
Next, we  multiply the equation (i) by $C_j^\top$ to obtain
\begin{align*}
kC_j+(\lambda-r)C_j^\top=(k\lambda -(r-1)\mu)J_v.
\end{align*} 
Noting that (ii) implies that each $C_i$ is symmetric, we have 
\begin{align*}
(k+\lambda-r)C_j=(k\lambda-(r-1)\mu)J_v.
\end{align*} 
Since none of $C_1,\ldots,C_r$ is the all-ones matrix, we obtain
\begin{align}
k+\lambda-r=0,\label{eq:p3}\\ 
k\lambda-(r-1)\mu=0. \label{eq:p4}
\end{align}
By letting $n=k/\mu$, the equations \eqref{eq:p1}, \eqref{eq:p2}, \eqref{eq:p3}, \eqref{eq:p4} show that 
\begin{align*}
v&=n^2\mu, \quad k=n\mu,\quad \lambda=\frac{n\mu-1}{n-1},\quad r=\frac{n^2\mu-1}{n-1}.
\end{align*}

It is indeed possible to characterize affine resolvable designs in terms of the auxiliary matrices as follows. 
Since $C_i$ is a symmetric $(0,1)$-matrix with all diagonal entries equal to one and $C_i^2=k C_i$, 
there exists some permutation matrix $P$ such that $C_i=P^\top (I_{v/k}\otimes J_k) P$.   
Define $k$-subsets $B_{i,1},\ldots,B_{i,v/k}$ of $\{1,\ldots,v\}$ such that $(C_i)_{xy}=1$ if and only if $\{x,y\}\subset B_{i,j}$ for some $j$. 
Then it is easy to see that the set of $B_{i,j}$'s forms an affine resolvable $2$-design with parallel classes $\{B_{i,1},\ldots,B_{i,v/k}\}$  $(i=1,\ldots,r)$.

Examples of affine resolvable designs include those obtained from Hadamard
matrices and Desarguesian geometries. Thanks to an anonymous referee for
pointing out to us that there are also  affine resolvable designs obtained 
from non-Desarguesian affine planes \cite{non-des-affine}. 
Next we introduce the auxiliary matrices corresponding to the first two families of affine resolvable designs.
%So far only two families of affine resolvable designs are known.  
%The auxiliary matrices corresponding to the two examples are shown in the next two examples. 
\begin{example}[Hadamard matrices]\label{ex:h}
Let $n\geq 4$ be an integer. 
Let $H$ be a Hadamard matrix of order $n$ with rows $r_0={\bf 1}^\top,r_1,\ldots,r_{n-1}$, where ${\bf 1}$ is the all-ones column vector.  
We define $C_i=\frac{1}{2}(r_i^\top r_i+J_n)$ for $i=1,\ldots,n-1$. 
Then:   
\begin{enumerate}
\item $\sum_{i=1}^{n-1}C_i=\frac{n}{2} I_{n}+\frac{n-2}{2}J_{n}$ holds. 
\item For any $i\in\{1,\ldots,n-1\}$, $C_iC_i^\top=\frac{n}{2} C_i$ holds. 
\item For any distinct $i,j\in\{1,\ldots,n-1\}$, $C_i C_j^\top=\frac{n}{4}J_{n}$ holds.
\end{enumerate}
\end{example}

\begin{example}[Affine geometries]\label{ex:ag}
Let $q$ be a prime power, $d$ a positive integer, and $V$ the $(d+1)$-dimensional vector space over the field $GF(q)$. 
We will call subspaces of dimension $d$ of $V$ hyperplanes, and their cosets $d$-flats. 
If $H$ is a hyperplane and $x,y\in V$, we will call $d$-flats $H+x$ and $H+y$ parallel. 
The space $V$ contains $r:=(q^{d+1}-1)/(q-1)$ hyperplanes, which we denote $H_1,\ldots,H_r$. 
All $d$-flats parallel to $H_i$ form a parallel class $\Pi_i$, $|\Pi_i|=q$.

%For $i=1,\ldots,r$, let $F_i$ be a $d$-flat parallel to the hyperplane $H_i$. 
%Let $F_{r+1}=\emptyset$. 
Define a $q^{d+1}\times q^{d+1}$ $(0,1)$-matrix $C_i$ with rows and columns indexed by the elements of $V$ by 
\begin{align*}
(C_i)_{x,y}=\begin{cases}
1 & \text{ if } y-x\in H_i,\\
0 & \text{ otherwise}.
\end{cases}
\end{align*}
Then:   
\begin{enumerate}
\item $\sum_{i=1}^{r}C_i= q^{2d} I_{q^{d+1}}+(r-1)q^{d-1} J_{q^{d+1}}$ holds. 
\item For any $i\in\{1,\ldots,r\}$, $C_i C_i^\top=q^d C_i$ holds. 
\item For any distinct $i,j\in\{1,\ldots,r\}$, $C_iC_j^\top=q^{d-1} J_{q^{d+1}}$ holds.
\end{enumerate}
\end{example}

%%%%%%%%%%%%%%%%%%%%%%%%%%%%%%%%%%%%%%%%%%%%%%%%%%%%%%%%%%%%%%%%%%%%%%%%%%%%%%%%%%%%%%%%%%%%%%

\subsection{Mutually Orthogonal Latin Squares (MOLS)}
We start this section with the variant of definition of orthogonal Latin squares used in this paper. 
A Latin square $L$ of size $n$ on the symbol set $\{1,\ldots,n\}$ is an $n\times n$ matrix with entries in $\{1,\ldots,n\}$ such that each row and each column has exactly one element in  $\{1,\ldots,n\}$. 
Two Latin squares $L_1$ and $L_2$ of size $n$ on the symbol set $\{1,\ldots,n\}$ are called to be 
{\it  orthogonal},  if every superimposition of each row of $L_1$ on each row of $L_2$ results in
only one element of the form $(a,a)$. In effect,  each permutation of symbols between the rows of the two Latin squares has a unique fixed symbol.
A set of Latin squares in which every distinct pair of Latin squares are orthogonal 
is called
\emph{Mutually Orthogonal   Latin Squares}(MOLS) \cite{HBCD}.
Note that different terminologies were used in \cite{HKO}, \cite{KS} and elsewhere for MOLS, depending on the way that they were used.

For the sake of completeness, we need the following two well known lemmas in order to introduce the following Notation~\ref{def:ls} below. 
\begin{lemma}\label{lem:sl}
Let  $L_1,L_2$ be MOLS on the set $\{1,\ldots,n\}$ with the $(i,j)$-entry equal to $l(i,j),l'(i,j)$ respectively.  
An $n\times n$ array with the $(i,j)$-entry equal to $b$ determined by $b=l(i,a)=l'(j,a)$ for the unique $a\in\{1,\ldots,n\}$, is a Latin square.
\end{lemma}
\begin{proof}
Routine. 
\end{proof}

\begin{lemma}\upshape{\cite[Lemma~2.5]{KS}}\label{lem:sl1}
Let $L_1,L_2,L_3$ be any MOLS on the set $\{1,\ldots,n\}$, and $L_{i,j}$ {\upshape(}$i,j\in\{1,2,3\},i\neq j${\upshape)} the Latin square obtained from $L_i$ and $L_j$ in this ordering by Lemma~\ref{lem:sl}. 
Then 
$L_{1,3}$ and $L_{2,3}$ are MOLS, and the Latin square obtained from $L_{1,3}$ and $L_{2,3}$ in this ordering equals to $L_{1,2}$. 
\end{lemma}
%\begin{proof}
%See \cite[Lemma~2.5]{KS}. 
%\end{proof}

Now we use the following notion on Latin squares, which was introduced in \cite{KS}. 
\begin{notation}\label{def:ls}
Let $f\geq 3$ be an integer.
Let $L_{i,j}$ ($i,j\in\{1,\ldots,f\},i\neq j$) be Latin squares on the same symbol set. 
Then $L_{i,j}$ ($i,j\in\{1,\ldots,f\},i\neq j$) are said to be \emph{linked MOLS} if for any distinct $i,j,k\in\{1,\ldots,f\}$, 
$L_{i,k}$ and $L_{j,k}$ are MOLS and the Latin square obtained from $L_{i,k}$ and $L_{j,k}$ in this ordering via Lemma~\ref{lem:sl} coincides with $L_{i,j}$.
\end{notation}

It is standard practice to obtain MOLS from finite fields. 
In the following construction, the resulting MOLS satisfy an additional condition described in Proposition~\ref{prop:mslsff} (ii) below. 
Let $\mathbb{F}_{q}$ be the finite field of $q$ elements $\alpha_1=0,\alpha_2,\ldots,\alpha_{q}$. 
Let $S$ be the subtraction table, i.e., $S=(\alpha_i-\alpha_j)_{i,j=1}^{q}$. 
For $k\in\{2,\ldots,q\}$, set $S_k=(\alpha_k(\alpha_i-\alpha_j))_{i,j=1}^{q}$. 
For distinct $k,k'\in\{2,\ldots,q\}$, let $S_{k,k'}$ denote the Latin square obtained from $S_k$ and $S_{k'}$ in this ordering.  
\begin{proposition}\upshape{\cite[Proposition~2.7]{KS}}\label{prop:mslsff}
\begin{enumerate}
\item The matrices $S_k$ {\upshape(}$k\in\{2,\ldots,q\}${\upshape)} are MOLS. 
\item If $q=2^n$, the Latin square obtained from $S_k$ and $S_{k,k'}$ is $S_{k'}$ for distinct $k,k'\in\{2,\ldots,2^n\}$.
\end{enumerate} 
\end{proposition}

%%%%%%%%%%%%%%%%%%%%%%%%%%%%%%%%%%%%%%%%%%%%%%%%%%%%%%%%%%%%%%%%%%%%%%%%%%%%%%%%%%
\section{Symmetric group divisible designs having partial complements}
In this section, we will deal with the incidence matrix $A$ of a symmetric group divisible design with parameters $(v,k,m,n,\lambda_1,\lambda_2)$ such that $A+K_{m,n}$ is a symmetric group divisible design with parameters $(v,k',m,n,\lambda'_1,\lambda'_2)$. 
Then the matrix $A'=J_v-K_{m,n}-A$ is a symmetric group divisible design with parameters $(v,v-k-n,m,n,v-n-2k+\lambda_1,v-2n-2k+\lambda_2+\frac{2k}{m-1})$, and is called a \emph{partial complement} of $A$.

The following lemma is needed later on for symmetric group divisible designs with the specific property mentioned below. 
This lemma shows that the parameters of $A+K_{m,n}$ are uniquely determined from the parameters of $A$ and that $A$ has some block structure.  
\begin{lemma}\label{lem:sd}
Assume that $m>1$ and $n>1$.  
Let $A$ be a symmetric group divisible design with parameters $(v,k,m,n,\lambda_1,\lambda_2)$, and assume that $A+K_{m,n}$ is  a symmetric group divisible design with parameters $(v,k',m,n,\lambda'_1,\lambda'_2)$.
Then $A K_{m,n}=K_{m,n}A=\frac{k}{m-1}(J_v-K_{m,n})$, $k'=k+n$, $\lambda'_1=\lambda_1+n$, and $\lambda'_2=\lambda_2+\frac{2k}{m-1}$. 
\end{lemma}
\begin{proof}
Counting the number of $1$ of $A$ and $A+K_{m,n}$ in each row, we have $k'=k+n$. 
Considering the inner product between two rows in a same group, we have $\lambda'_1=\lambda_1+n$. 
Next, we calculate $(A+K_{m,n})(A^\top +K_{m,n})$ in two ways. 
On the one hand, since $A$ is a symmetric group divisible design with parameters $(v,k,m,n,\lambda_1,\lambda_2)$, one has
\begin{align}
(A+K_{m,n})(A^\top+K_{m,n})&=A A^\top +A K_{m,n}+K_{m,n}A^\top+(K_{m,n})^2\nonumber\\
&=A K_{m,n}+K_{m,n}A^\top+(k-\lambda_1)I_v+(\lambda_1+n-\lambda_2)K_{m,n}+\lambda_2 J_v.\label{eq:ak1}
\end{align}
On the other hand, since $A+K_{m,n}$ is  a symmetric group divisible design with parameters $(v,k',m,n,\lambda'_1,\lambda'_2)$, one has
\begin{align}\label{eq:ak2}
(A+K_{m,n})(A^\top+K_{m,n})&=(k'-\lambda'_1)I_v+(\lambda'_1-\lambda'_2)K_{m,n}+\lambda'_2 J_v\nonumber\\
&=(k-\lambda_1)I_v+(\lambda_1+n-\lambda'_2)K_{m,n}+\lambda'_2 J_v. 
\end{align} 
Equations \eqref{eq:ak1} and \eqref{eq:ak2} show 
\begin{align}
A&K_{m,n}+K_{m,n}A^\top=(\lambda'_2-\lambda_2)(J_v-K_{m,n}).\label{eq:1}
\end{align}
Since the diagonal blocks of $A$ are the zero matrix, 
$A K_{m,n}$ and $K_{m,n}A^\top$ have zero diagonal blocks. 
%Note that since the matrix $A K_{m,n}+K_{m,n}A^\top$ is not a zero matrix, $\lambda'_2-\lambda_2>0$. 
Furthermore, if we let $A=(B_{i,j})_{i,j=1}^m$ where each $B_{i,j}$ is an $n\times n$ matrix, then Equation~\eqref{eq:1} implies that $B_{i,j}J_n+J_n B_{j,i}^\top=(\lambda'_2-\lambda_2) J_n$ for any distinct $i,j$. 
Therefore $B_{i,j}J_n=J_nB_{i,j}=\frac{\lambda'_2-\lambda_2}{2}J_n$ for any distinct $i,j$. 
Thus we have 
\begin{align*}
A K_{m,n}=K_{m,n}A=\frac{\lambda'_2-\lambda_2}{2}(J_v-K_{m,n}). 
\end{align*}

Finally, since the number of $1$'s in each row of  $A$ is $k$, we have $\lambda'_2-\lambda_2=\frac{2k}{m-1}$. 
\end{proof}

\begin{proposition}\label{prop:psd}
Assume that $m>1$ and $n>1$.  
Let $A$ be a symmetric group divisible design with parameters $(v,k,m,n,\lambda_1,\lambda_2)$ such that $A+K_{m,n}$ is a symmetric group divisible design. 
Then:
\begin{align*}
\lambda_1=\frac{k(k-m+1)}{(m-1)(n-1)},\quad \lambda_2=\frac{k^2(m-2)}{n(m-1)^2}.
\end{align*}
\end{proposition}
\begin{proof} 
Define the standard inner product in the matrix algebra:
$$
(M,N ) = \text{Tr}(MN^\top)=\sum_{i,j}M_{i,j}N_{i,j}, \quad (M,N\in\text{Mat}_v(\mathbb{R})).
$$

It follows from $AK_{m,n} =\frac{k}{m-1}(J_v-K_{m,n})$ by Lemma~\ref{lem:sd} that
 $(AK_{m,n},AK_{m,n})=\frac{k^2 v(v - n)}{(m-1)^2}$. 
On the other hand,
\begin{align*}
&(AK_{m,n},AK_{m,n}) = \text{Tr}((AK_{m,n})(AK_{m,n})^\top) = \text{Tr}(AK_{m,n}^2A^\top) =n\text{Tr}(A^\top AK_{m,n})\\
&=n(A^\top A,K_{m,n}) = n(kI_v + \lambda_1(K_{m,n}- I_v) + \lambda_2(J_v - K_{m,n}),K_{m,n}) = n(kv + \lambda_1(n - 1)v). 
\end{align*}
Therefore, $\frac{k^2 v(v - n)}{(m-1)^2} = n(kv + \lambda_1(n - 1)v)$. 
After cancellation with $v=mn$, we obtain the
formula for $\lambda_1$ given in Proposition~\ref{prop:psd}. The constant $\lambda_2$ is derived from Equation~\eqref{eq:p1}.
\end{proof}

\section{Linked systems of symmetric group divisible designs of type II}\label{sec:lsgdd}
%Let $K_{m,n}=I_m\otimes J_n$. 
Let $f\geq 3$ be an integer. 
Let $A_{i,j}$ ($i,j\in\{1,\ldots,f\},i\neq j$) be symmetric group divisible designs with the parameters $(v,k,m,n,\lambda_1,\lambda_2)$. 
The symmetric group divisible designs $A_{i,j}$ ($i,j\in\{1,\ldots,f\},i\neq j$) are said to be a \defn{linked system of symmetric group divisible designs of type II} if the following conditions are satisfied; 
\begin{enumerate}
\item for any distinct $i,j\in\{1,\ldots,f\}$, $A_{i,j}+K_{m,n}$ is a $(0,1)$-matrix, and 
\item there exist non-negative integers $\sigma,\tau,\rho$ such that for any distinct $i,j$ and $l$, $A_{i,j}A_{j,l}=\sigma A_{i,l}+\tau(J_v-A_{i,l}-K_{m,n})+\rho K_{m,n}$ holds. 
\end{enumerate} 
%Partial complements $A'_{i,j}$ ($i,j\in\{1,\ldots,f\},i\neq j$) satisfy the above with parameters $(\sigma',\tau',\rho')$ given as 
%\begin{align*}
%\sigma'&=\tau+v-2n-2k+\frac{2k}{m-1},\\
%\tau'&=\sigma+v-2n-2k+\frac{2k}{m-1},\\
%\rho'&=\rho+v-n-2k.
%\end{align*}

\begin{lemma}\label{lem:rt}
Let $A_{i,j}$ {\rm(}$i,j\in\{1,\ldots,f\},i\neq j${\rm)} be a linked system of symmetric group divisible designs of type II with parameters $(v,k,m,n,\lambda_1,\lambda_2)$ and $\sigma,\tau,\rho$. 
If $k>\lambda_1$, then $\rho\neq \tau$. 
\end{lemma}
\begin{proof}
Assume on the contrary that $\rho=\tau$. 
Then the matrices $A_{i,j}$ satisfy that $A_{i,j}A_{j,l}=\sigma A_{i,l}+\tau(J_v-A_{i,l})$. Thus $A_{i,j}$ ($i,j\in\{1,\ldots,f\},i\neq j$) form a linked system of symmetric group divisible designs \cite{KS}.  
From \cite[Lemma~3.3(ii)]{KS}\footnote{In \cite{KS} the theorems are valid under the assumption $k>\lambda_1$. If $k=\lambda_1$, then linked systems of symmetric group divisible designs are $A_{i,j}\otimes J_n$ where $A_{i,j}$'s are linked systems of symmetric designs.} and $k>\lambda_1$, it follows that $A_{i,j}K_{m,n}=\alpha J_v$ for some integer $\alpha$.   
We assumed that $A_{i,j}+K_{m,n}$ is a $(0,1)$-matrix, which implies that the diagonal blocks of $A_{i,j}$ are the zero matrices. So $\alpha=0$, and thus $k=0$, a contradiction to $k>\lambda_1\geq 0$.  
\end{proof}
Next we show that the matrix $A_{i,j}+K_{m,n}$ is a symmetric group divisible design for any distinct $i,j$ and we obtain equalities involving $\sigma,\tau,\rho$. 
\begin{lemma}\label{lem:1}
Assume that $m>1,n>1$, $\lambda_1<k<(m-1)n$ and $f\geq 3$. 
Let $A_{i,j}$ {\rm(}$i,j\in\{1,\ldots,f\},i\neq j${\rm)} be a linked system of symmetric group divisible designs of type II with parameters $(v,k,m,n,\lambda_1,\lambda_2)$ and $\sigma,\tau,\rho$. 
Then the following hold:
\begin{enumerate}
\item  $(\sigma-\tau)^2=k-\lambda_1$. 
\item $(\sigma-\tau)(\rho-\tau)+(\sigma-\tau+k)\tau=k \lambda_2$. 
\item $A_{i,j}K_{m,n}=K_{m,n}A_{i,j}=\frac{k}{m-1}(J_v-K_{m,n})$ and $A_{i,j}+K_{m,n}$ is a symmetric group divisible design for any distinct $i,j$. 
\item  $\left(\rho-\tau-\lambda_1+\lambda_2\right)\frac{k}{m-1}=(\sigma-\tau)(\rho-\tau)$. 
%\item  $-\frac{k(\lambda_1-\lambda_2)}{m-1}=(\rho-\tau)(\sigma-\tau-\frac{k}{m-1})$, and $\frac{k(\lambda_1-\lambda_2)}{m-1}+k \lambda_2=(\sigma-\tau)\tau+\tau k+\frac{k(\rho-\tau)}{m-1}$.
\end{enumerate}
\end{lemma}
\begin{proof}
Let $i,j,l$ be mutually distinct integers in $\{1,2,\ldots,f\}$.  
Calculate $(A_{i,j}A_{j,l})A_{l,j}=A_{i,j}(A_{j,l}A_{l,j})$ in two ways as follows.
On the one hand, 
\begin{align}
(A_{i,j}A_{j,l})A_{l,j}&=(\sigma A_{i,l}+\tau (J_v-A_{i,l}-K_{m,n})+\rho K_{m,n})A_{l,j} \nonumber\\
&=(\sigma-\tau) A_{i,l}A_{l,j}+(\rho-\tau)K_{m,n}A_{l,j}+\tau J_vA_{l,j} \nonumber\\
&=(\sigma-\tau)(\sigma A_{i,j}+\tau (J_v-A_{i,j}-K_{m,n})+\rho K_{m,n})+(\rho-\tau)K_{m,n}A_{l,j}+\tau k J_v \nonumber\\
&=(\sigma-\tau)^2A_{i,j}+(\sigma-\tau)(\rho-\tau)K_{m,n}+(\sigma-\tau+k)\tau J_v+(\rho-\tau)K_{m,n}A_{l,j}\label{eq:aijl1},
\end{align} 
and on the other hand 
\begin{align}
A_{i,j}(A_{j,l}A_{l,j})&=A_{i,j}(k I_v+\lambda_1(K_{m,n}-I_v)+\lambda_2(J_v-K_{m,n}))\nonumber\\
&=(k-\lambda_1) A_{i,j}+(\lambda_1-\lambda_2)A_{i,j}K_{m,n}+\lambda_2 A_{i,j}J_v\nonumber\\
&=(k-\lambda_1) A_{i,j}+(\lambda_1-\lambda_2)A_{i,j}K_{m,n}+k\lambda_2 J_v. \label{eq:aijl2}
\end{align} 

Equating \eqref{eq:aijl1} and \eqref{eq:aijl2}, one has 
\begin{align}\label{eq:aijl3}
(\sigma-\tau)^2A_{i,j}+(\sigma-\tau)(\rho-\tau)K_{m,n}+(\sigma-\tau+k)\tau J_v+(\rho-\tau)K_{m,n}A_{l,j}\nonumber\\
=(k-\lambda_1) A_{i,j}+(\lambda_1-\lambda_2)A_{i,j}K_{m,n}+k\lambda_2 J_v.
\end{align}
Simplifying Equation \eqref{eq:aijl3}, and multiplying on the left and right by $X\otimes (I_n-\frac{1}{n}J_n)$, where $X$ is an arbitrary $m\times m$ matrix, on both sides, one has
\begin{align*} 
((\sigma-\tau)^2-k+\lambda_1)\left(X\otimes (I_n-\frac{1}{n}J_n)\right)A_{i,j}\left(X\otimes (I_n-\frac{1}{n}J_n)\right)=O. 
\end{align*}
Since $A_{i,j}$ is not the zero matrix, there exist $a,b$ such that the $(a,b)$-block of $A_{i,j}$ is not the zero matrix. Also by $\lambda_1<k<(m-1)n$, it follows that the $(a,b)$-block of $A_{i,j}$ is not the all-ones matrix.   
Taking $X$ as $E_{a,b}$, that is the $m\times m$ $(0,1)$-matrix whose $(x,y)$-entry equals to $1$ if and only if $(x,y)=(a,b)$, 
$\left(E_{a,b}\otimes (I_n-\frac{1}{n}J_n)\right)A_{i,j}\left(E_{a,b}\otimes (I_n-\frac{1}{n}J_n)\right)$ is not the  zero matrix. 
Thus $(\sigma-\tau)^2-k+\lambda_1=0$.  This proves (i).    

Therefore  Equation \eqref{eq:aijl3} yields that 
\begin{align}\label{eq:aijl4}
(\rho-\tau)K_{m,n}A_{l,j}+(\lambda_2-\lambda_1)A_{i,j}K_{m,n}=-(\sigma-\tau)(\rho-\tau)K_{m,n}+(k\lambda_2-(\sigma-\tau+k)\tau) J_v.
\end{align}
Since $A_{i,j}$'s have the zero diagonal blocks, comparing the diagonal blocks of Equation~\eqref{eq:aijl4} provides 
$(\sigma-\tau)(\rho-\tau)+(\sigma-\tau+k)\tau=k \lambda_2$. 
This proves (ii). 

Thus Equation~\eqref{eq:aijl4} is simplified as 
\begin{align}\label{eq:aijl5}
(\rho-\tau)K_{m,n}A_{l,j}+(\lambda_2-\lambda_1)A_{i,j}K_{m,n}=(\sigma-\tau)(\rho-\tau)(J_v-K_{m,n}).
\end{align}
Then it is easy to see that $(\rho-\tau)K_{m,n}A_{l,j}$ and $(\lambda_2-\lambda_1)A_{i,j}K_{m,n}$ are multiples of $J_v-K_{m,n}$. 
Since $\rho\neq \tau$ by Lemma~\ref{lem:rt}, $K_{m,n}A_{l,j}$ is also a multiple of $J_v-K_{m,n}$. 
Similarly, the same argument for $A_{i,j}A_{j,i}A_{i,l}$ yields that $A_{l,j}K_{m,n}$ is a multiple of $J_v-K_{m,n}$. 
It is easy to see that $A_{i,j}K_{m,n}=K_{m,n}A_{i,j}=\frac{k}{m-1}(J_v-K_{m,n})$. 
Therefore it follows that $A_{i,j}+K_{m,n}$ is a symmetric group divisible design. 
This proves (iii). 

Substituting the fact that $A K_{m,n}=K_{m,n}A=\frac{k}{m-1}(J_v-K_{m,n})$ for any distinct $i,j$ into \eqref{eq:aijl5}, we have 
%\begin{align}\label{eq:aijl6}
%\left(\rho-\tau-\lambda_1+\lambda_2\right)\frac{k}{m-1}(J_v-K_{m,n})=(\sigma-\tau)(\rho-\tau)(J_v-K_{m,n}).
%\end{align}
 $\left(\rho-\tau-\lambda_1+\lambda_2\right)\frac{k}{m-1}(J_v-K_{m,n})=(\sigma-\tau)(\rho-\tau)(J_v-K_{m,n})$, which proves (iv). 
\end{proof}

We now derive the formula for $\sigma,\tau,\rho$ as a consequence of Proposition~\ref{prop:psd} and Lemma~\ref{lem:1}. 
\begin{proposition}\label{prop:str}
Assume that $m>1,n>1$, $\lambda_1<k<(m-1)n$ and $f\geq 3$. 
Let $A_{i,j}$ {\rm(}$i,j\in\{1,\ldots,f\},i\neq j${\rm)} be a linked system of symmetric group divisible designs of type II with parameters $(v,k,m,n,\lambda_1,\lambda_2)$ and $\sigma,\tau,\rho$. 
%\begin{enumerate}
%\item If $\lambda_1=\lambda_2$, then the following hold: 
%\begin{align*}
%\sigma&=\frac{n(m-1)(m^2-3m+2\pm(n-1))}{(m+n-2)^2},\quad \tau=\frac{n(m-1)(m^2-3m+2\mp(m-1))}{(m+n-2)^2},\\
%\rho&=\frac{n(m-1)^3}{(m+n-2)^2}. 
%\end{align*}
%\item 
%If $\lambda_1\neq \lambda_2$, then 
Then the following hold: 
\begin{align*}
\sigma&=\frac{k^2 (m-2) (n-1)\pm(mn -k - 
 n) \sqrt{k (m-1) (n-1) (m n-k-n)}}{(m-1)^2 (n-1) n},\\ 
\tau&=\frac{k^2 (m-2) (n-1)\mp k\sqrt{k (m-1) (n-1) (m n-k-n)}}{(m-1)^2 (n-1) n},\\
\rho&=\frac{k^2}{n(m-1)}. 
\end{align*}
%\end{enumerate}
\end{proposition}
\begin{proof}
The formula for $\sigma,\tau,\rho$ follows from Lemma~\ref{lem:1} with Proposition~\ref{prop:psd}. 
\end{proof}
In general, there are two possibilities for a pair $(\sigma,\tau)$ for given $k,m$, and $n$. 
The following lemma eliminates one of the cases for the case $\lambda_1=\lambda_2$. 
\begin{proposition}
Assume that $m>1,n>1$, $\lambda_1<k<(m-1)n$ and $f\geq 3$. 
Let $A_{i,j}$ {\rm(}$i,j\in\{1,\ldots,f\},i\neq j${\rm)} be a linked system of symmetric group divisible designs of type II with parameters $(v,k,m,n,\lambda_1,\lambda_2)$ and $\sigma,\tau,\rho$. 
If $\lambda_1=\lambda_2$, then $(\sigma,\tau,\rho)=(\frac{(m-1) n (m^2-3 m+1+n)}{(m+n-2)^2},\frac{(m-3) (m-1)^2 n}{(m+n-2)^2},\frac{(m-1)^3 n}{(m+n-2)^2})$. 
\end{proposition}
\begin{proof}
It follows from Proposition~\ref{prop:str} and Lemma~\ref{lem:rt}.
\end{proof}

For the case $\lambda_1=\lambda_2$, there exist several examples with $f\geq 3$. 
We will provide an upper bound on $f$ in Section~\ref{sec:as} and some constructions in Section~\ref{sec:ex}.

%%%%%%%%%%%%%%%%%%%%%%%%%%%%%%%%%%%%%%%%%%%%%%%%%%%%%%%%%%%%%%%%%%%%%%%%%%%%%%%%%%%%%%%%%%%%%%%%%%%%%%%%%%%%%%%%%%%
\section{Association schemes}\label{sec:as}
In this section we obtain a symmetric association schemes with $5$-classes from any linked system of symmetric group divisible designs of type II and characterize a linked system of symmetric group divisible designs of type II in terms of such association schemes. 

Let $A_{i,j}$ ($i,j\in\{1,\ldots,f\},i\neq j$) be a linked system of symmetric group divisible designs of type II with parameters $(v,k,m,n,\lambda_1,\lambda_2)$ and $\sigma,\tau,\rho$. 
Set 
\begin{align}
A_0&=I_{f mn}, \quad A_1=I_{f m}\otimes (J_{n}-I_{n}), \quad A_2=I_f\otimes( J_{m}-I_{m})\otimes J_{n}, \label{eq:am1}\\
A_3&=\begin{pmatrix} 
O_{mn} & A_{1,2} & \cdots & A_{1,f} \\
A_{2,1} & O_{mn}  & \cdots & A_{2,f} \\
\vdots & \vdots & \ddots & \vdots \\
A_{f,1} & A_{f,2}  & \cdots & O_{mn}
\end{pmatrix}, 
A_4=(J_f-I_f)\otimes J_{m n}-A_3-A_5, \label{eq:am2}\\
A_5&=(J_f-I_f)\otimes I_m\otimes J_n.\label{eq:am3}
\end{align}
Note that $A_3=O_{fmn}$ if and only if $k=0$ and that $A_4=O_{fmn}$ if and only if $k=(m-1)n$. 

\begin{theorem}
Let $A_{i,j}$ {\rm(}$i,j\in\{1,\ldots,f\},i\neq j${\rm)} be symmetric group divisible designs. Assume one of the following: 
\begin{enumerate}
\item $f=2$, $m>1,n>1$, $\lambda_1<k<(m-1)n$ and $A_{1,2}$ satisfies $A_{1,2}+K_{m,n}$ is a symmetric group divisible design. 
\item $f\geq 3$, $m>1,n>1$, $\lambda_1<k<(m-1)n$ and $A_{i,j}$ {\rm(}$i,j\in\{1,\ldots,f\},i\neq j${\rm)} are a linked system of symmetric group divisible designs of type II with parameters $(v,k,m,n,\lambda_1,\lambda_2)$ and $\sigma,\tau,\rho$. 
\end{enumerate}
Then the set of matrices $\{A_0,A_1,\ldots,A_5\}$ is a symmetric association scheme with  $5$-classes. 
\end{theorem}
\begin{proof}
It is enough to show that for any $i$ and $j$, $A_iA_j$ is a linear combination of $A_0,A_1,\ldots,A_5$. Since $A_0+A_1=I_{f m}\otimes J_{n}$,  $A_0+A_1+A_2=I_f\otimes J_{m n}$ and $A_3+A_4+A_5=I_f\otimes (J_{m}-I_{m})\otimes J_{n}$ hold, we will show that the product of any two in $\{I_{f m}\otimes J_{n},I_{f}\otimes J_{m n},J_{f m n},A_3,A_5\}$ belongs to $\mathcal{A}$. 

The cases $(A_3)^2,A_3 A_5,(A_5)^2$ follow from the definition of a linked system of symmetric group divisible designs of type II, Lemma~\ref{lem:1}, and direct calculation, respectively.
The other cases follow from direct calculation and Lemma~\ref{lem:1}.  
\end{proof}
We list the eigenmatrices $P,Q$, the Krein matrix $B_2^*$ by \cite[Theorem~3.6 (ii)]{BI} as follows. 
\begin{align}
P&=\left(
\begin{array}{cccccc}
 1 & n-1 & (m-1) n & (f-1) k & (f-1) (m n-k-n) & (f-1) n \\
 1 & -1 & 0 & \frac{(f-1) \sqrt{k (m n -k-n)}}{\sqrt{(m-1) (n-1)}} & -\frac{(f-1) \sqrt{k (m n -k-n)}}{\sqrt{(m-1) (n-1)}} & 0 \\
 1 & n-1 & -n & -\frac{(f-1) k}{m-1} & \frac{(f-1) (k+n-m n)}{m-1} & (f-1) n \\
 1 & n-1 & -n & \frac{k}{m-1} & -\frac{k}{m-1}+n & -n \\
 1 & -1 & 0 & -\frac{\sqrt{k (m n -k-n)}}{\sqrt{(m-1) (n-1)}} & \frac{\sqrt{k (m n -k-n)}}{\sqrt{(m-1) (n-1)}} & 0 \\
 1 & n-1 & (m-1) n & -k & k+n-m n & -n \\
\end{array}
\right),\label{eq:P}\displaybreak[0]\\
Q&=\left(
\begin{array}{cccccc}
 1 & m (n-1) & m-1 & (f-1) (m-1) & (f-1) m (n-1) & f-1 \\
 1 & -m & m-1 & (f-1) (m-1) & -(f-1)m & f-1 \\
 1 & 0 & -1 & -f+1 & 0 & f-1 \\
 1 & \frac{m \sqrt{(n-1) (m n-k-n)}}{\sqrt{k (m-1)}} & -1 & 1 & -\frac{m \sqrt{(n-1) (m n-k-n)}}{\sqrt{k(m-1)}} & -1 \\
 1 & -\frac{\sqrt{k m (n-1)}}{\sqrt{m n-k-n}} & -1 & 1 & \frac{\sqrt{k m (n-1)}}{\sqrt{m n-k-n}} & -1 \\
 1 & 0 & m-1 & -m+1 & 0 & -1 \\
\end{array}
\right),\label{eq:Q}\displaybreak[0]\\
B_2^*&=\left(
\begin{array}{cccccc}
 0 & 0 & 1 & 0 & 0 & 0 \\
 0 & \frac{m}{f}-1 & 0 & 0 & \frac{m}{f} & 0 \\
 m-1 & 0 & m-2 & 0 & 0 & 0 \\
 0 & 0 & 0 & m-2 & 0 & m-1 \\
 0 & \frac{(f-1) m}{f} & 0 & 0 & m-1-\frac{m}{f} & 0 \\
 0 & 0 & 0 & 1 & 0 & 0 \\
\end{array}
\right).\nonumber
\end{align}

\begin{corollary}
$f\leq m$ holds.
\end{corollary}
\begin{proof}
The Krein parameter $q_{2,1}^1$ equals $m/f-1$. Proposition~\ref{prop:krein} implies that $f\leq m$.  
\end{proof}

Next we characterize linked systems of symmetric group divisible designs of type II in terms of the imprimitivity of the corresponding association scheme. 
\begin{theorem}\label{thm:aslgdd}
Let $\{A_0,A_1,\ldots,A_5\}$ be a symmetric association scheme and the eigenmatrices $P,Q$ as in \eqref{eq:P}, \eqref{eq:Q}.  
Set $f=m_5+1$, $\lambda_1=\frac{1}{f-1}p_{3,3}^1$ and  $\lambda_2=\frac{1}{f-1}p_{3,3}^4$. 
Assume that $m>1,n>1$, $\lambda_1<k<(m-1)n$. 
%Then the following holds. 
\begin{enumerate}
\item If $f=2$, then there exists a symmetric group divisible design $A_{1,2}$ with parameters $(m n,k,m,n,\lambda_1,\lambda_2)$ such that $A_{1,2}+K_{m,n}$ is a symmetric group divisible design.
\item If $f\geq 3$, then there exists a linked system of symmetric group divisible designs of type II $A_{i,j}$ {\upshape (}$i,j\in\{1,\ldots,f\},i\neq j${\upshape )} with parameters $(m n,k,m,n,\lambda_1,\lambda_2)$. %such that $A_{i,j}K_{m,n}=K_{m,n}A_{i,j}=\frac{p_{3,3}^0}{(f-1)m}(J_v-K_{m,n})$.  
\end{enumerate}
\end{theorem}
\begin{proof}
From the fact that a strongly regular graph having two distinct eigenvalues is a disjoint union of the complete graphs of the same order, it follows that $A_1=I_{fm}\otimes (J_n-I_n)$ and $A_1+A_2=I_f\otimes(J_{mn}-I_{mn})$. 
Thus 
\begin{align*}
A_2=I_f\otimes( J_{m}-I_{m})\otimes J_{n}. 
\end{align*}

Next we determine $A_5$. 
Let $A_5=(Z_{ij})_{i,j=1}^{f}$ where each $Z_{ij}$ is an $mn\times mn$ matrix and $Z_{ii}=O$ for any $i$. 
Since the intersection numbers are calculated by \cite[Theorem~3.6 (i)]{BI}, we have 
\begin{align}\label{eq:a5}
(A_0+A_1)A_5=n A_5, \quad
(A_0+A_1+A_2)A_5=n(A_3+A_4+A_5). 
\end{align} 
Let $i,j$ be fixed distinct integers in $\{1,\ldots,f\}$, and we write $Z=Z_{ij}$. 
Further decompose the matrix $Z$ as $Z=(z_{kl})_{k,l=1}^{m}$ where each $z_{kl}$ is an $n\times n$ matrix. 
By \eqref{eq:a5}, we have for any $k$
\begin{align*}
J_{n}z_{kk}=n z_{kk},\quad \sum_{l=1}^n J_n z_{kl}=n z_{kk}. 
\end{align*}
Since each row of $z_{kl}$ has constant sum $n$, 
\begin{align*}
z_{kk}=J_n,\quad z_{kl}=O_n \text{ for any } l\neq k. 
\end{align*}
Thus $Z_{ij}=Z=K_{m,n}$, and $A_5=(J_f-I_f)\otimes I_m\otimes J_n$.

Now we may write 
\begin{align*}
A_3=\begin{pmatrix} 
O_{mn} & A_{1,2} & \cdots & A_{1,f} \\
A_{2,1} & O_{mn}  & \cdots & A_{2,f} \\
\vdots & \vdots & \ddots & \vdots \\
A_{f,1} & A_{f,2}  & \cdots & O_{mn}
\end{pmatrix}.
\end{align*}
By \cite[Theorem~3.6 (ii)]{BI}, the Krein matrix $B_5^*$ is 
\begin{align*}
B_5^*=
\begin{pmatrix}
 0 & 0 & 0 & 0 & 0 & 1 \\
 0 & 0 & 0 & 0 & 1 & 0 \\
 0 & 0 & 0 & 1 & 0 & 0 \\
 0 & 0 & f-1 & f-2 & 0 & 0 \\
 0 & f-1 & 0 & 0 & f-2 & 0 \\
 f-1 & 0 & 0 & 0 & 0 & f-2 \\
\end{pmatrix}.
\end{align*}
Thus the association scheme is uniform by \cite[Proposition~4.7, Corollary~4.15]{DMM}. 
Let $X_1,\ldots,X_f$ be the equivalence classes with respect to the equivalence relation corresponding to $A_0+A_1+A_2+A_5$. 
Then, for any subset $I$ of $\{1,\ldots,f\}$ and $Y=\cup_{i\in I}X_i$, the set of matrices obtained from $A_0,A_1,\ldots,A_5$ by restricting to $Y$ is an association scheme. 

For any distinct $i,j\in\{1,\ldots,f\}$ ($i<j$), set $I=\{i,j\}$. 
Then the adjacency matrices of the association scheme with the vertex set $Y=X_i\cup X_j$ are  
\begin{align*}
A^Y_0&=I_{2mn},\quad A^Y_1=I_{2m}\otimes (J_n-I_n),\\
A^Y_2&=I_2\otimes( J_{m}-I_{m})\otimes J_{n},\quad A^Y_5=(J_2-I_2)\otimes I_m\otimes J_n,\\
A^Y_3&=\begin{pmatrix} 
O_{mn} & A_{i,j}\\
A_{j,i} & O_{mn}
\end{pmatrix},\quad A^Y_4=\begin{pmatrix} 
O_{mn} & J_{m n}-K_{m,n}-A_{i,j}\\
J_{m n}-K_{m,n}-A_{j,i} & O_{mn}
\end{pmatrix}.
\end{align*}
Then $(A^Y_3)^2$ is written as a linear combination of $A^Y_0,\ldots,A^Y_5$.  
Comparing the diagonal blocks of the equation, we have that $A_{i,j}$ is a symmetric group divisible design. 
Furthermore, by looking at $A^Y_3+A^Y_5$, we conclude that $A_{i,j}+K_{m,n}$ is a $(0,1)$-matrix. 
 
Next if $f\geq3$, then we consider the case for $I=\{i,j,l\}$, where $i,j,l\in\{1,\ldots,f\}$ ($i<j<l$) are mutually distinct, and for $Y=X_i\cup X_j\cup X_l$.  
Then the adjacency matrices of the association scheme with the vertex set $Y=X_i\cup X_j\cup X_l$ are  
\begin{align*}
A^Y_0&=I_{3mn},\quad A^Y_1=I_{3m}\otimes (J_n-I_n),\\ 
A^Y_2&=I_3\otimes( J_{m}-I_{m})\otimes J_{n},\quad A^Y_5=(J_3-I_3)\otimes I_m\otimes J_n,\\
A^Y_3&=\begin{pmatrix} 
O_{mn} & A_{i,j} & A_{i,l} \\
A_{j,i} & O_{mn} & A_{j,l} \\
A_{l,i} & A_{l,j} & O_{mn} 
\end{pmatrix},\\ 
A^Y_4&=\begin{pmatrix} 
O_{mn} & J_{m n}-K_{m,n}-A_{i,j} & J_{m n}-K_{m,n}-A_{i,l} \\
J_{m n}-K_{m,n}-A_{j,i} & O_{mn} & J_{m n}-K_{m,n}-A_{j,l} \\
J_{m n}-K_{m,n}-A_{l,i} & J_{m n}-K_{m,n}-A_{l,j} & O_{mn} 
\end{pmatrix}.
\end{align*}
Then $(A^Y_3)^2$ is written as a linear combination of $A^Y_0,\ldots,A^Y_5$.  
Comparing the $(1,3)$ block of the equation, we have that $A_{i,j}A_{j,l}$ is a linear combination of $A_{i,l},K_{m,n},J_{m n}$, where the coefficients do not depend on the choice of $i,j,l$.  
Thus $A_{i,j}$'s are a linked system of symmetric group divisible designs of type II. 
\end{proof}

\begin{remark}
\begin{enumerate}
\item Consider the case $\lambda_1'=\lambda_2'$, that is $k=\frac{(m-1)n(n-1)}{n+m-2}$ by Proposition~\ref{prop:psd}.  Then from the Bannai-Muzychuk criterion \cite{B,M} for fusion schemes it follows that $\{A_0,A_1+A_2,A_3+A_5,A_4\}$ is a fusion scheme with the primitive idempotents $\{E_0,E_1,E_2+E_5,E_3+E_4\}$. 
Conversely, if the set of adjacency matrices $\{A_0,A_1+A_2,A_3+A_5,A_4\}$ is a fusion scheme, 
then the only possibility of a partition of the set of primitive idempotents is $\{E_0,E_1,E_2+E_5,E_3+E_4\}$ with $k=\frac{(m-1)n(n-1)}{n+m-2}$. Furthermore, $\lambda_1'=\lambda_2'$ holds by Proposition~\ref{prop:psd}.   
In this case, $A_{i,j}+K_{m,n}$ ($i,j\in\{1,\ldots,f\},i\neq j$) are a linked system of symmetric designs, that is a $Q$-antipodal $Q$-polynomial association scheme with $3$-classes \cite{M,Dam}.  
\item Association schemes for the case $f=2$ and $k=m-1$ were studied in \cite{QDK}.  
\end{enumerate}
\end{remark}
%%%%%%%%%%%%%%%%%%%%%%%%%%%%%%%%%%%%%%%%%%%%%%%%%%%%%%%%%%%%%%%%%%%%%%%%%%%%%%%%%%%
\section{Construction of Symmetric Group Divisible Designs, SGDD}\label{sec:con}
In this section we provide constructions of symmetric group divisible designs $A$ with the properties that $A+K_{m,n}$ is a $(0,1)$-matrix and that $AK_{m,n}=K_{m,n}A=\alpha(J-K_{m,n})$ for some positive integer $\alpha$. 
We first begin with a construction using conference matrices. 
A \defn{conference matrix of order $n$}, $n$ even, is an $n\times n$ $(0,1,-1)$-matrix $C$ with diagonal $0$ such that $C C^\top=(n-1)I$. 
%The order $n$ of a conference matrix must be even. 
\begin{proposition}
Let $C$ be a conference matrix of order $n$. 
Let $A$ be the $(0,1)$-matrix obtained by replacing $1$ {\rm(}$-1$, $0$ respectively{\rm)} in $C$ with $I_2$ {\rm(}$J_2-I_2$, $O_2$ respectively{\rm)}. 
Then $A$ is a symmetric group divisible design with the parameters $(2n,n-1,n,2,0,n/2-1)$ and the properties that $A+K_{n,2}$ is a $(0,1)$-matrix and that $AK_{n,2}=K_{n,2}A=J_{2n}-K_{n,2}$.  
\end{proposition}
\begin{proof}
Straightforward. 
\end{proof}
We generalize this construction in two ways. 
In Section~\ref{sec:gc} we make use of generalized conference matrices, and in Section~\ref{sec:dw} we make  use of disjoint weighing matrices.  

\subsection{Generalized conference matrices}\label{sec:gc}
Let $G$ be a multiplicatively written finite group of order $g\geq2$. 
A \emph{generalized conference matrix with parameters $(g, \lambda)$ over $G$} is a $(g\lambda+2)\times (g\lambda+2)$ matrix $C = (c_{ij})_{i,j=1}^{g\lambda+2}$ with entries from $G \cup \{0\}$ such that (i) every row of $W$ contains exactly 
$g\lambda+1$ nonzero entries and the diagonal entries are $0$ and (ii) for any distinct $i, h \in \{1, 2,\ldots, g\lambda+2\}$, every element of $G$ is
contained exactly $\lambda$ times in the multiset $\{c_{ij} c^{-1}_{hj} \mid 1 \leq j \leq g\lambda+2, c_{ij} \neq 0, c_{hj}\neq 0\}$.

Let $\phi$ be a permutation representation of $G$. 
We extend the domain of $\phi$ to $G\cup\{0\}$ so that $\phi(0)=O_v$. 
Then we construct a symmetric group divisible design from any generalized conference matrix. 
\begin{proposition}
Let $C = (c_{ij})_{i,j=1}^{g\lambda+2}$ be a generalized conference matrix $GC(g,\lambda)$ over $G$. 
Then the matrix $\phi(C):=(\phi(c_{ij}))_{i,j=1}^{g\lambda+2}$ is a symmetric group divisible design with parameters $(g(g\lambda+2),g\lambda+1,g\lambda+2,g,0,\lambda)$ satisfying that $\phi(C)+K_{g\lambda+2,g}$ is a $(0,1)$-matrix and that $\phi(C)K_{g\lambda+2,g}=K_{g\lambda+2,g}\phi(C)=J_{g(g\lambda+2)}-K_{g\lambda+2,g}$.  
\end{proposition} 
\begin{proof}
%Since the diagonal entries of $C$ are $0$, $\phi(C)+K_{g\lambda+2,g}$ is a $(0,1)$-matrix. 
Compute $\phi(C)\phi(C)^\top$ as follows: for $i,j\in\{1,\ldots,g\lambda+2\}$, 
\begin{align*}
\text{the $(i,j)$-block of }\phi(C)\phi(C)^\top&=\sum_{k=1}^{g\lambda+2}\phi(c_{ik})\phi(c_{jk})^\top\\
&=\sum_{k=1}^{g\lambda+2}\phi(c_{ik}c_{jk}^{-1})\\
&=\begin{cases}
\sum_{1\leq k\leq g\lambda+2,k\neq i}\phi(e)= (g\lambda+1) I_g & \text{ if } i=j,\\
\lambda\sum_{x\in G}\phi(x)=\lambda J_g & \text{ if } i\neq j. 
\end{cases}
\end{align*}
Thus we have $\phi(C)\phi(C)^\top=(g\lambda+1) I_{g(g\lambda+2)}+\lambda(J_{g(g\lambda+2)}-K_{g\lambda+2,g})$.
Similarly, we have the formula for $\phi(C)^\top\phi(C)$. 
Thus $\phi(C)$ is the incidence matrix of a symmetric group divisible design with the desired parameters.
Next, since the diagonal entries of $C$ is $0$, $\phi(C)+K_{g\lambda+2,g}$ is a $(0,1)$-matrix.
Finally, each off-diagonal block of $\phi(C)$ is a permutation matrix, $\phi(C)K_{g\lambda+2,g}=K_{g\lambda+2,g}\phi(C)=J_{g(g\lambda+2)}-K_{g\lambda+2,g}$.     
\end{proof}
\begin{example}
A generalized conference matrix with parameters $(g,\lambda)$ is called a \emph{balanced generalized weighing matrix with parameters $(g\lambda+2,g\lambda+1,g\lambda)$}. 
For any prime power $q$, there exists a balanced generalized weighing matrix with parameters $(q+1,q,q-1)$ \cite[Lemma~3]{KT}.   
\end{example}
This case cannot be extended to a linked system of symmetric group divisible designs of type II as is shown below. 
If there exists a linked system of symmetric group divisible designs of type II with parameters $(g(g\lambda+2),g\lambda+1,g\lambda+2,g,0,\lambda)$ and $f\geq 3$, then by Proposition~\ref{prop:str}, 
\begin{align*}
\sigma=\lambda\pm\frac{(g-1)\sqrt{g\lambda+1}}{g},\quad \tau=\lambda\mp\frac{\sqrt{g\lambda+1}}{g},\quad \rho=\lambda+\frac{1}{g}, 
\end{align*}
which shows that $\rho$ is an integer if and only if $g=1$, and this is impossible.

\subsection{Twin symmetric GDDs}\label{sec:dw}
Two symmetric group divisible designs $A^+,A^-$ are said to be \defn{twin} if $A^++A^-$ is a $(0,1)$-matrix. 
A \defn{weighing matrix of order $n$ and weight  $k$} is an $n\times n$ $(0,1,-1)$-matrix $W$ such that $W W^\top=kI$. 
Note that conference matrices of order $n$ are weighing matrices of order $n$ and weight $n-1$. 

\begin{proposition}\label{prop:dw}
Let $n,m$ be positive integers such that there exist a Hadamard matrix of order $n$ and weighing matrices $W_1,\ldots,W_{n-1}$ of order $(n-1)m+1$ and weight $m$ satisfying $\sum_{i=1}^{n-1}|W_i|=J-I$, where $|W_i|$ is the matrix obtained by replacing $-1$ with $1$ in $W_i$. 
There exist twin symmetric GDDs with parameters $(n((n-1)m+1),\frac{n(n-1)m}{2},(n-1)m+1,n,\frac{n(n-2)m}{4},\frac{n((n-1)m-1)}{4})$ satisfying that $A^++A^-+K_{(n-1)m+1,n}=J$ and $A^+K_{(n-1)m+1,n}=K_{(n-1)m+1,n}A^+=A^-K_{(n-1)m+1,n}=K_{(n-1)m+1,n}A^-=\frac{n}{2}(J-K_{(n-1)m+1,n})$. 
\end{proposition}
\begin{proof}
Set $\ell=(n-1)m+1$.  
Let $C_1,\ldots,C_{n-1}$ be the auxiliary $(0,1)$-matrices of a normalized Hadamard matrix of order $n$. 

For each weighing matrix $W_i$, we write $W_i=W_{i,1}-W_{i,2}$ where  $W_{i,1}$ and $W_{i,2}$ are disjoint $(0,1)$-matrices. 
The equation $W_i W_i^\top=m I$ implies that 
\begin{align*}
W_{i,1}W_{i,1}^\top+W_{i,2}W_{i,2}^\top-(W_{i,1}W_{i,2}^\top+W_{i,2}W_{i,1}^\top)=m I. 
\end{align*}
Define $A^+,A^-$ to be  
\begin{align*}
A^+&=\sum_{i=1}^{n-1}(W_{i,1}\otimes C_i+W_{i,2}\otimes (J-C_i)),\\
A^-&=\sum_{i=1}^{n-1}(W_{i,1}\otimes (J_n-C_i)+W_{i,2}\otimes C_i).
\end{align*}
Our claim is that $A^+$ and $A^-$ are twin symmetric group divisible designs. 
First, we show that $A^+$ is a symmetric GDD:  
\begin{align*}
A^+(A^+)^\top&=\sum_{i,j=1}^{n-1}(W_{i,1}W_{j,1}^\top\otimes C_i C_j^\top+W_{i,2}W_{j,2}^\top\otimes (J_n-C_i) (J_n-C_j^\top)\\
&\quad \quad+W_{i,1}W_{j,2}^\top\otimes C_i (J_n-C_j^\top)+W_{i,2}W_{j,1}^\top\otimes (J_n-C_i) C_j^\top)\displaybreak[0]\\
&=\sum_{i=1}^{n-1}(W_{i,1}W_{i,1}^\top\otimes C_i C_i^\top+W_{i,2}W_{i,2}^\top\otimes (J_n-C_i) (J_n-C_i^\top)\displaybreak[0]\\
&\quad \quad +W_{i,1}W_{i,2}^\top\otimes C_i (J_n-C_i^\top)+W_{i,2}W_{i,1}^\top\otimes (J_n-C_i) C_i^\top)\displaybreak[0]\\
&+\sum_{i\neq j}(W_{i,1}W_{j,1}^\top\otimes C_i C_j^\top+W_{i,2}W_{j,2}^\top\otimes (J_n-C_i) (J_n-C_j^\top)\displaybreak[0]\\
&\quad \quad +W_{i,1}W_{j,2}^\top\otimes C_i (J_n-C_j^\top)+W_{i,2}W_{j,1}^\top\otimes (J_n-C_i) C_j^\top)\displaybreak[0]\\
%&=\frac{n}{2}\sum_{i=1}^{n-1}(W_{i,1}W_{i,1}^\top\otimes C_i+W_{i,2}W_{i,2}^\top\otimes C_i +W_{i,1}W_{i,2}^\top\otimes (J_n-C_i)+W_{i,2}W_{i,1}^\top\otimes (J_n-C_i))\\
%&+\frac{n}{4}\sum_{i\neq j}(W_{i,1}W_{j,1}^\top\otimes J+W_{i,2}W_{j,2}^\top\otimes J+W_{i,1}W_{j,2}^\top\otimes J+W_{i,2}W_{j,1}^\top\otimes J)\\
&=\frac{n}{2}\sum_{i=1}^{n-1}((W_{i,1}W_{i,1}^\top+W_{i,2}W_{i,2}^\top-W_{i,1}W_{i,2}^\top-W_{i,2}W_{i,1}^\top)\otimes C_i +(W_{i,1}W_{i,2}^\top+W_{i,2}W_{i,1}^\top)\otimes J_n)\displaybreak[0]\\
&\quad \quad +\frac{n}{4}\sum_{i\neq j}(W_{i,1}W_{j,1}^\top+W_{i,2}W_{j,2}^\top+W_{i,1}W_{j,2}^\top+W_{i,2}W_{j,1}^\top)\otimes J_n\displaybreak[0]\\
&=\frac{nm}{2}\sum_{i=1}^{n-1} I_{\ell}\otimes C_i +\frac{n}{4}\sum_{i,j=1}^{n-1}(W_{i,1}+W_{i,2})(W_{j,1}^\top+W_{j,2}^\top)\otimes J_n-\frac{nm}{4}\sum_{i=1}^{n-1} I_{\ell}\otimes J_n\displaybreak[0]\\
&=\frac{nm}{2} I_{\ell}\otimes (\frac{n}{2}I_n+\frac{n-2}{2}J_n) +\frac{n}{4}(J_{\ell}-I_{\ell})^2\otimes J_n-\frac{n(n-1)m}{4}  I_{\ell}\otimes J_n\displaybreak[0]\\
&=\frac{n^2m}{4}  I_{\ell}\otimes I_n +\frac{n(-m+1)}{4}I_{\ell}\otimes J_n +\frac{n((n-1)m-1)}{4}J_{\ell}\otimes J_n. 
\end{align*}
Similarly, we have $(A^+)^\top A^+=\frac{n^2m}{4}  I_{\ell}\otimes I_n +\frac{n(-m+1)}{4}I_{\ell}\otimes J_n +\frac{n((n-1)m-1)}{4}J_{\ell}\otimes J_n$. 
Thus $A^+$ is a symmetric GDD with the desired parameters. 
The same calculation is applied to $A^-$ to obtain that $A^-$ is a symmetric GDD. 

Next, It is easy to see that $A^++A^-+K_{(n-1)m+1,n}=\sum_{i=1}^{n-1}(W_{i,1}+W_{i,2})\otimes J_n+K_{\ell,n}=(J_\ell-I_\ell)\otimes J_n+K_{\ell,n}=J_{\ell n}$. 
Finally, each off-diagonal block of $A^+$ has constant row and column sum $n/2$, $A^+ K_{\ell,n}=K_{\ell,n}A^+=\frac{n}{2}(J_{\ell n}-K_{\ell,n})$. Also the same is true for $A^-$. 
%Therefore $A^+,A^-$ are twin symmetric GDD satisfying $A^++K_{(n-1)m+1,n}$ and $A^-+K_{(n-1)m+1,n}$ are $(0,1)$-matrices. 
\end{proof}

We provide two examples of orthogonal designs satisfying the assumption of Proposition~\ref{prop:dw}.  
An \emph{orthogonal design of order $n$ and type $(s_1,\ldots,s_u)$ in variables $x_1,\ldots,x_u$} is an $n\times n$ $(0,\pm x_1,\ldots,\pm x_u)$-matrix $D$, where $x_1,\ldots,x_u$ are distinct commuting indeterminates, such that $D D^\top=(s_1x_1^2+\cdots+s_u x_u^2)I_n$.
%We denote it by {\it $OD(n;s_1,\ldots,s_u)$}. 

\begin{example}\label{ex:4t}
It is shown in \cite{R} that for $t>3$, there exists an orthogonal design of order $2^t$ and type $(s_i)_{i=1}^{2t}=(1,1,1,1,2,2,4,4,\ldots,2^{t-2},2^{t-2})$. 
By equating some variables and a variable corresponding to weight $1$ with $0$, there exists an orthogonal design of order $4^t$ and type $((4^t-1)/3,(4^t-1)/3,(4^t-1)/3)$ for any $t>0$. 
This provides  an example of disjoint weighing  matrices of order $4^t$ and weight $(4^t-1)/3$. 
%Note that a Hadamard matrix of order $4^t$ exists for any $t$. 
\end{example}

\begin{example}
An orthogonal design of order $28$ and type $(1,9,9,9)$ was constructed in \cite{Ko}. 
By equating the variable corresponding to weight $1$ with $0$, we have an orthogonal design of order $28$ and type $(9,9,9)$. 
This shows the existence  of disjoint weighing  matrices of order $28$ and weight $9$. 
%Note that a Hadamard matrix of order $28$ exists.
\end{example}

Further constructions of disjoint  weighing matrices will be given in a forthcoming paper \cite{KSpre}.

If there exists a linked system of symmetric group divisible designs of type II with parameters $(n((n-1)m+1),\frac{n(n-1)m}{2},(n-1)m+1,n,\frac{n(n-2)m}{4},\frac{n((n-1)m-1)}{4})$ and $f\geq 3$ where $n$ is the order of a Hadamard matrix, then by Proposition~\ref{prop:str}, 
\begin{align*}
\sigma=\frac{n((n-1)m+1\pm\sqrt{m})}{4},\quad \tau=\frac{n((n-1)m-1\pm\sqrt{m})}{4},\quad \rho=\frac{n(n-1)m}{4}.  
\end{align*}
Since $\sigma-\tau=\pm n\sqrt{m}/4$, $m$ must be a square of an integer. 
The twin designs constructed in Example~\ref{ex:4t} has the parameter $m=(4^t-1)/3$, which cannot be a square of an integer.   
%Therefore we conclude that the matrices $I_\ell\otimes I_n,K_{\ell,n}-I_\ell\otimes I_n,A^+,A^-$ are a directed association scheme.  

%If the weighing matrices $W_1,\ldots,W_{n-1}$ are all symmetric, then the set of matrices $I_\ell\otimes I_n,K_{\ell,n}-I_\ell\otimes I_n,A^+,A^-$ are a symmetric association scheme,. 
%If those $W_1,\ldots,W_{n-1}$ are all skew-symmetric, then the set of matrices $I_\ell\otimes I_n,K_{\ell,n}-I_\ell\otimes I_n,A^+,A^-$ are a non-symmetric association scheme. 

%%%%%%%%%%%%%%%%%%%%%%%%%%%%%%%%%%%%%%%%%%%%%%%%%%%%%%%%%%%%%%%%%%%%%%%%%%%%%%%%%%%

\section{Construction of linked systems of SGDDs of type II}\label{sec:ex}
In this section we show some examples of linked systems of symmetric group divisible designs of type II. 

\subsection{Mutually unbiased Bush-type Hadamard matrices}\label{sec:mubh}
The first example is from mutually unbiased Bush-type Hadamard matrices in \cite{KSS}. 

A Hadamard matrix $H$ of order $4n^2$ is of \defn{Bush-type} if $H=(H_{ij})_{i,j=1}^{2n}$ where each $H_{ij}$ is a $2n \times 2n$ matrix satisfies that $H_{ii}=J_{2n}$ for any $i$ and $H_{ij}J_{2n}=J_{2n} H_{ij}=O$ for any distinct $i,j$. 
Hadamard matrices $H_1$ and $H_2$ of order $4n^2$ are \defn{unbiased} if $\frac{1}{2n}H_1H_2^\top$ is a Hadamard matrix of order $4n^2$. 

Let $H_1,\ldots,H_f$ be mutually unbiased Bush-type-Hadamard matrices of order $4n^2$. 
Define $H_{i,0}=H_{0,i}^\top=H_i$ ($i\in\{1,\ldots,f\}$), $H_{i,j}=\frac{1}{2n}H_i H_j^\top$ ($i,j\in\{1,\ldots,f\},i\neq j$), and $A_{i,j}=\frac{1}{2}(-H_{i,j}+J_{4n^2})$ for $i,j\in\{0,1,\ldots,f\},i\neq j$. 
Then it is easy to see that $A_{i,j}$ ($i,j\in\{0,1,\ldots,f\},i\neq j$) are symmetric designs and form a linked system of symmetric group divisible designs of type II with parameters $(v,k,m,n,\lambda_1,\lambda_2)=(4n^2,2n^2-n,2n,2n,n^2-n,n^2-n)$ and $(\sigma,\tau,\rho)=(n^2-\frac{n}{2},n^2-\frac{3n}{2},n^2-\frac{n}{2})$.

\subsection{Affine resolvable designs}
Next we use affine resolvable designs and linked MOLS.
Let $C_1,\ldots,C_r$ be auxiliary matrices satisfying (i), (ii) and (iii), and  
$L$ a Latin square on $\{0,1,\ldots,r\}$ with $(i,j)$-entry denoted $l(i,j)$. 
We define $\tilde{L}$ by 
\begin{align*}
\tilde{L}=(C_{l(i,j)})_{i,j=0}^{r}.
\end{align*}

\begin{lemma}
The matrix $\tilde{L}$ is a symmetric $2$-$((r+1)v,k r, k \lambda)$ design.  
\end{lemma}

Let $L_{i,j}$ ($i,j\in\{1,\ldots,f\},i\neq j$) be linked MOLS Latin squares on $\{0,1,\ldots,r\}$ with diagonal entires $0$. 
We set $A_{i,j}=\tilde{L}_{i,j}$. %, and define $B_{i,j}=P_{i,j}\otimes J_g$ where $P_{i,j}$ is a permutation matrix obtained by replacing the entry  $0$ in $L_{i,j}$ with $1$ and other entries with $0$. 
Then the symmetric designs $A_{i,j}$ satisfy the following. 
\begin{proposition}\label{prop:lsd1}
For any distinct $i,j$ and $l$, 
$A_{i,j}A_{j,l}=k A_{i,l}+2\mu K_{r+1,v}+(r-2)\mu J_{v}$. 
\end{proposition}
\begin{proof}
We denote the $(\alpha,\beta)$-entry of $L_{i,j}$ and $L_{j,k}$ by $l(\alpha,\beta),l'(\alpha,\beta)$ respectively. 
Let $a,b$ be integers in the set $\{0,1,\ldots,r\}$. 
Let $c,d$ be the integers in $\{0,1,\ldots,r\}$ such that the entries of $a$-th row of $L_{i,j}$ and $b$-th row of $L_{j,l}$ share the entry $d$ at the $c$-th position.    

When $d=0$, the $(a,b)$-block of $A_{i,j}A_{j,l}$ is 
\begin{align*}
\sum_{0\leq\alpha\leq r,\alpha\neq c} C_{l(a,\alpha)}C_{l'(b,\alpha)}^\top=r\mu J_v. 
\end{align*}
When $d\neq 0$, the $(a,b)$-block of $A_{i,j}A_{j,l}$ is 
\begin{align*}
C_d C_d^\top +\sum_{0\leq\alpha\leq r,\alpha\neq c} C_{l(a,\alpha)}C_{l'(b,\alpha)}^\top=k C_d+(r-2)\mu J_v. 
\end{align*}
Thus we obtain the desired equation. 
\end{proof}
In Table~\ref{Tab:Par1} we list the feasible parameters with $v\leq 1000$ for linked systems of symmetric group divisible designs of type II where one of the designs or their partial complements is a collection of symmetric designs.  
Combining Examples~\ref{ex:h}, \ref{ex:ag} and Proposition~\ref{prop:mslsff}, we have Table~\ref{Tab:Par1}.

In contrast, Table~\ref{Tab:Par2} is about feasible parameters where both the designs and their partial complements do not form a set of symmetric designs.
%We close the this paper with some interesting problems which are worth looking into as future works listed below. 
We could not find any example of linked systems of symmetric group divisible designs of type II with parameters appearing in Table~\ref{Tab:Par2}. 
We close this paper with the following problem.
\begin{problem}
Does there exist a linked system of symmetric group divisible designs of type II where both the designs and their partial complements do not form a set of symmetric designs?
\end{problem}
%%%%%%%%%%%%%%
\section*{Acknowledgement}
The authors would like to thank the referees for many useful comments, especially for suggesting a change that significantly shortened the proof of Proposition 3.2.
Hadi Kharaghani is supported by an NSERC Discovery Grant.  
Sho Suda is supported by JSPS KAKENHI Grant Number 15K21075, 18K03395. 

%\pagebreak

\begin{table}[H]
\caption{Feasible parameters $(v,k,\lambda,m,n,\sigma,\tau,\rho)$ with $v\leq 1000$ for symmetric designs}
\label{Tab:Par1}
\begin{center}
%{\small
%{\footnotesize
{\scriptsize
%\begin{tabular}{c|c|l|ccccccc}
\begin{tabular}{c|c|c|c}
\noalign{\hrule height0.8pt}
$(v,k,\lambda)$ & $(m,n,\sigma,\tau,\rho)$ & Existence & Bounds on $f$      \\
\hline  
$(16,6,2)$ & $(4,4,3,1,3)$& Yes  & $3\leq f\leq 4$  \\
$(45,12,3)$ & $(5,9,5,2,4)$& Yes  & $4\leq f\leq 5$  \\
$(64,28,12)$ & $(8,8,14,10,14)$& Yes  & $7\leq f\leq 8$  \\
$(96,20,4)$ & $(6,16,7,3,5)$& Yes & $1\leq f\leq 6$  \\
$(144,66,30)$ & $(12,12,33,27,33)$& Yes & $4\leq f\leq 12$  \\
$(153,96,60)$ & $(17,9,62,56,64)$&  &   \\
$(175,30,5)$ & $(7,25,9,4,6)$& Yes & $6\leq f\leq 7$  \\
$(256,120,56)$ & $(16,16,60,52,60)$& Yes  & $15\leq f\leq 16$  \\
$(288,42,6)$ & $(8,36,11,5,7)$&  &   \\
$(378,117,36)$ & $(14,27,42,33,39)$& Yes & $1\leq f\leq 14$  \\
$(400,190,90)$ & $(20,20,95,85,95)$& Yes & $3\leq f\leq 20$ \\
$(425,160,60)$ & $(17,25,66,56,64)$&  &   \\
$(441,56,7)$ & $(9,49,13,6,8)$& Yes & $8\leq f\leq 9$  \\
$(576,276,132)$ & $(24,24,138,126,138)$& Yes  & $2\leq f\leq 24$  \\
$(630,408,264)$ & $(35,18,268,256,272)$&  &   \\
$(640,72,8)$ & $(10,64,15,7,9)$& Yes & $1\leq f\leq 10$  \\
$(736,540,396)$ & $(46,16,399,387,405)$&  &   \\
$(784,378,182)$ & $(28,28,189,175,189)$& Yes & $3\leq f\leq 28$  \\
$(891,90,9)$ & $(11,81,17,8,10)$& Yes & $10\leq f\leq 11$ \\
$(925,540,315)$ & $(37,25,321,306,324)$&  &  \\
\noalign{\hrule height0.8pt}
\end{tabular}
}
\end{center}
\end{table}

\begin{table}[H]
\caption{Feasible parameters $(v,k,\lambda,m,n,\sigma,\tau,\rho)$ with $v\leq 500$ for proper symmetric group divisible designs with their partial complements being proper symmetric group divisible designs}
\label{Tab:Par2}
\begin{center}
%{\small
%{\footnotesize
{\scriptsize
%\begin{tabular}{c|c|l|ccccccc}
\begin{tabular}{ccccccccc||cccccccccc}
$v$ & $k$ & $m$ & $n$ & $\lambda_1$ & $\lambda_2$ & $\sigma$ & $\tau$ & $\rho$   &$v$ & $k$ & $m$ & $n$ & $\lambda_1$ & $\lambda_2$ & $\sigma$ & $\tau$ & $\rho$ \\
\hline  
 52 & 24 & 13 & 4 & 8 & 11 & 9 & 13 & 12 & 52 & 24 & 13 & 4 & 8 & 11 & 13 & 9 & 12 \\
 112 & 54 & 28 & 4 & 18 & 26 & 23 & 29 & 27& 112 & 54 & 28 & 4 & 18 & 26 & 29 & 23 & 27 \\
 125 & 40 & 5 & 25 & 15 & 12 & 9 & 14 & 16 & 125 & 40 & 5 & 25 & 15 & 12 & 15 & 10 & 16 \\
 125 & 60 & 5 & 25 & 35 & 27 & 25 & 30 & 36 & 125 & 60 & 5 & 25 & 35 & 27 & 29 & 24 & 36 \\
 196 & 96 & 49 & 4 & 32 & 47 & 43 & 51 & 48 & 196 & 96 & 49 & 4 & 32 & 47 & 51 & 43 & 48 \\
 232 & 112 & 29 & 8 & 48 & 54 & 50 & 58 & 56 & 232 & 112 & 29 & 8 & 48 & 54 & 58 & 50 & 56 \\
 245 & 84 & 5 & 49 & 35 & 27 & 23 & 30 & 36 & 245 & 84 & 5 & 49 & 35 & 27 & 31 & 24 & 36 \\
 245 & 112 & 5 & 49 & 63 & 48 & 45 & 52 & 64 & 245 & 112 & 5 & 49 & 63 & 48 & 51 & 44 & 64 \\
 304 & 150 & 76 & 4 & 50 & 74 & 69 & 79 & 75 & 304 & 150 & 76 & 4 & 50 & 74 & 79 & 69 & 75 \\
 333 & 108 & 37 & 9 & 27 & 35 & 29 & 38 & 36 & 333 & 108 & 37 & 9 & 27 & 35 & 41 & 32 & 36 \\
 333 & 216 & 37 & 9 & 135 & 140 & 137 & 146 & 144 & 333 & 216 & 37 & 9 & 135 & 140 & 143 & 134 & 144 \\
 336 & 80 & 21 & 16 & 16 & 19 & 13 & 21 & 20 & 336 & 80 & 21 & 16 & 16 & 19 & 25 & 17 & 20 \\
 336 & 240 & 21 & 16 & 176 & 171 & 169 & 177 & 180 & 336 & 240 & 21 & 16 & 176 & 171 & 173 & 165 & 180 \\
 405 & 144 & 5 & 81 & 63 & 48 & 43 & 52 & 64 & 405 & 144 & 5 & 81 & 63 & 48 & 53 & 44 & 64 \\
 405 & 180 & 5 & 81 & 99 & 75 & 71 & 80 & 100 & 405 & 180 & 5 & 81 & 99 & 75 & 79 & 70 & 100 \\
 436 & 216 & 109 & 4 & 72 & 107 & 101 & 113 & 108 & 436 & 216 & 109 & 4 & 72 & 107 & 113 & 101 & 108 
\end{tabular}
}
\end{center}
\end{table}

%\pagebreak

%%%%%%%%%%%%%%%%%%%%%%%%%%%%%%%%%%%%%%%%%%%%%%%%%%%%%%%%%%%%%%%%%%%%%%%%%%%%%%%%%%%
%%%%%%%%%%%%%%%%%%%%%%%%%%%%%%%%%%%%%%%%%%%%%%%%%%%%%%%%%%%%%%%%%%%%%%%%%%%%%%%%%%
%%%%%%%%%%%%%%%%%%%%%%%%%%%%%%%%%%%%%%%%%%%%%

\end{document}